\documentclass[11pt]{amsart}
\usepackage{amsfonts,amssymb,amsthm}
\usepackage[numbers,sort&compress]{natbib}
\usepackage[leqno]{amsmath}
\usepackage[mathscr]{euscript}
\usepackage{color}
\usepackage{paralist}
\usepackage{hyperref}
\usepackage{framed}
\usepackage{caption}
\newtheorem {theorem}{Theorem}[section]
\newtheorem {proposition}{Proposition}[section]

\newtheorem {lemma}{Lemma}[section]
\newtheorem {example}{Example}[section]
\newtheorem {definition}{Definition}[section]
\newtheorem {remark}{Remark}[section]

\title{The Charnes--Cooper transformation and fractional programs with convex polynomials}

\author{Chengmiao Yang}
\address[Chengmiao Yang]{Academy for Advanced Interdisciplinary Studies, Northeast Normal University, Changchun 130024, Jilin Province, People's Republic of China}
\email{cmyang@nenu.edu.cn}

\author{Liguo Jiao}
\address[Liguo Jiao]{Academy for Advanced Interdisciplinary Studies, Northeast Normal University, Changchun 130024, Jilin Province, People's Republic of China}
\email{jiaolg356@nenu.edu.cn; hanchezi@163.com}

\author{Jae Hyoung Lee$^{\dag}$}
\address[Jae Hyoung Lee]{Department of Applied Mathematics, Pukyong National University, Busan 48513, Korea}
\email{mc7558@naver.com}

\thanks{$^\dag$Corresponding Author}

\date{\today}

\begin{document}

\begin{abstract}
This paper proposes a scheme based on the Charnes--Cooper transformation for solving a class of fractional programs with convex polynomial data.
We employ the high-degree perturbation method to construct relaxations based on sums of squares for the fractional programs under consideration.
Under suitable conditions, we establish results on solution existence, strong duality, asymptotic convergence, and solution extraction.
As a special case, fractional programs with SOS-convex polynomial data are also studied.
\end{abstract}

\keywords{Fractional programs; Charnes--Cooper variable transformation; convex polynomials; high-degree perturbation method; semidefinite programming.}

\maketitle

\section{Introduction}\label{sect:1}

Fractional programming is an important class of nonlinear optimization in which the objective function is expressed as a ratio of two functions.
Such models arise naturally in economics, engineering management, and operations research, where the objective often represents a ratio between two performance measures; see, e.g., \cite{Bajalinov2003,Schaible2003,StancuMinisian2012}.
A general fractional program can be written as
\begin{equation}\label{equ0}
r^*:=\inf_{x\in\Omega}\frac{f(x)}{g(x)},
\end{equation}
where $\Omega$ is a nonempty subset of $\mathbb{R}^n$ and the denominator $g$ is assumed to be positive on $\Omega$.
A notable feature of problem~\eqref{equ0} is that the ratio $f/g$ is generally nonconvex, even when $f$ is convex and $g$ is concave.

One of the most classical and influential approaches to fractional programming is the Dinkelbach method, introduced by Dinkelbach~\cite{Dinkelbach1967}.
Under standard assumptions, the optimal value $r^*$ is characterized by
\begin{equation*}
\inf_{x\in\Omega}\left\{f(x)-r^*g(x)\right\}=0.
\end{equation*}
This characterization leads to an iterative procedure in which, at the $k$th iteration, one solves the non-fractional subproblem
\begin{equation*}
\min_{x\in\Omega}\left\{f(x)-r_k g(x)\right\}
\end{equation*}
and updates the parameter according to
\begin{equation*}
r_{k+1}:=\frac{f(x_{k+1})}{g(x_{k+1})}.
\end{equation*}
Thus, the Dinkelbach method replaces the fractional problem with a sequence of parameterized non-fractional subproblems.
Although this framework is effective for many structured fractional programs, each iteration still requires solving an optimization subproblem that can itself be computationally demanding.
Several extensions of the Dinkelbach-type method have been developed; see, e.g., \cite{MR4810562}.

Another classical approach is the Charnes--Cooper transformation (CCT), which reformulates a fractional program in terms of lifted variables; see \cite{MR152370,Chen2005JOTA,schaible1974}.
Unlike Dinkelbach-type methods, the CCT produces a single reformulated problem and does not require an iteratively updated ratio parameter.
In this sense, the CCT reformulation is \emph{parameter-free}.
Under suitable convexity assumptions, the resulting problem can be expressed as a convex optimization problem involving perspective functions, at the cost of introducing an additional variable and additional constraints arising from the transformation.

First-order iterative methods have also been developed for fractional programs with nonsmooth or structured data.
Proximal-gradient and proximal-subgradient algorithms are particularly suited to problems involving nonsmooth terms \cite{Bot2017,Bot2}, whereas splitting and block-coordinate methods can exploit separable or decomposable structures in the numerator and denominator \cite{Bot1,Bot3}.
These approaches differ from the Dinkelbach method and the CCT in that they directly use first-order information and proximal operations to generate a sequence of approximate solutions.

For fractional programs with quadratic or polynomial data, semidefinite programming methods provide another important class of approaches.
Nguyen, Sheu, and Xia~\cite{Nguyen2015} developed an SDP approach for quadratic fractional programs.
Jiao and Lee~\cite{Jiao2019} established exact SDP relaxations for fractional optimization problems involving support functions.
Guo and Jiao~\cite{Guo2021} studied fractional semi-infinite polynomial programming problems and developed conic and SDP relaxations with asymptotic convergence.
Guo and Zhang~\cite{Guo2024} subsequently proposed an SDP method for fractional semi-infinite programs with SOS-convex polynomials.
More recently, Yang, Jiao, and Lee~\cite{Yang2024} considered fractional programs with SOS-convex semi-algebraic functions and developed a parameter-free SDP-based approach together with a solution
extraction result.

The present paper studies the finitely constrained fractional polynomial program
\begin{equation*}
\inf_{x\in\mathbb{R}^n} \left\{\frac{f(x)}{g(x)} \mid h_i(x)\leq 0, \ i=1,\ldots,m \right\},
\end{equation*}
where $f$, $-g$, and $h_i$, $i=1,\ldots,m$, are convex polynomials.
In contrast to the SOS-convex setting considered in \cite{Yang2024}, the polynomial data in the present paper are not required to be SOS-convex.
Our setting is therefore more general with respect to polynomial convexity, although it is restricted to polynomial functions rather than general SOS-convex semi-algebraic functions.

Removing the SOS-convexity assumption gives rise to an important difficulty.
After applying the CCT and forming the Lagrange dual, one obtains a global nonnegativity constraint on a convex polynomial.
A globally nonnegative convex polynomial, however, need not be a sum of squares.
Consequently, the exact one-step SDP reformulation available in the SOS-convex setting cannot generally be applied.

To overcome this difficulty, we combine the convex perspective reformulation obtained from the CCT with the high-degree perturbation method of Lasserre and Netzer~\cite{Lasserre2007}.
The perturbation method approximates a globally nonnegative polynomial by sums of squares and consequently yields a hierarchy of SDP relaxations.
We establish the asymptotic convergence of this hierarchy and derive a moment-type dual problem from which an optimal solution can be extracted under a rank-one condition.
When all the polynomial data are SOS-convex, the perturbation term is unnecessary and the general hierarchy reduces to an exact SDP relaxation at the first level.

\subsection*{Contributions}

Our main contributions, relative to the above literature, are summarized as follows.
\begin{enumerate}[\upshape (i)]
\item By combining the Charnes--Cooper transformation with the Lasserre--Netzer high-degree perturbation method, we construct a hierarchy of SDP relaxations for fractional programs in which $f$, $-g$, and $h_i$, $i=1,\ldots,m$, are convex polynomials that need not be SOS-convex.
\item Under the nonemptiness and positivity assumptions on the feasible set, we establish the existence of an optimal solution provided that the denominator $g$ is bounded above on the feasible set.
We also prove the existence of an optimal solution to the corresponding CCT reformulation and show that the two problems have the same optimal value.
\item Under the Slater condition, we establish strong duality between the CCT reformulation and its Lagrange dual.
When the feasible set is bounded, we prove the asymptotic convergence of the proposed SDP hierarchy.
We further derive a moment-type dual problem and provide a rank-one sufficient condition for finite convergence and solution extraction.
\item When $f$, $-g$, and $h_i$, $i=1,\ldots,m$, are SOS-convex, we show that the perturbation term is unnecessary and that the first SDP relaxation is exact.
Moreover, an optimal solution can be extracted from an optimal moment solution satisfying $\bar y_0\neq 0$.
\end{enumerate}

The organization of this paper is as follows.
Section~\ref{sect::2} introduces the notation and preliminary results used throughout the paper.
Section~\ref{sect::3} presents the Charnes--Cooper transformation and establishes the relationships among the original fractional problem and its perspective reformulations.
Section~\ref{sect::4} contains our main results, including the existence of optimal solutions, the Lagrange dual formulation, the proposed SDP relaxation hierarchy, its asymptotic convergence, and the
solution extraction result.
An illustrative example is also provided in this section.
The SOS-convex special case is discussed in Section~\ref{sos-convex}.
Finally, Section~\ref{sect::5} concludes the paper.

\section{Preliminaries}\label{sect::2}

\subsection{Basic notation}
Here, we collect some notation and preliminary results that will be used in this paper.
Throughout the paper, the space $\mathbb{R}^{n}$ is equipped with the standard Euclidean norm, i.e., $\|x\| = \sqrt{x_{1}^{2} + \cdots + x_{n}^{2}}$ for all $x := \left(x_{1}, \ldots, x_{n}\right) \in \mathbb{R}^{n}.$
Let $n$ be a positive integer, and let $\mathbb{N}$ denote the set of nonnegative integers.
$\mathbb{R}_{+}^{n}$ denotes the nonnegative orthant:
\begin{equation*}
\mathbb{R}_{+}^{n} := \left\{x \in \mathbb{R}^{n} \mid x_{i} \geq 0,\ i = 1, \ldots, n\right\}.
\end{equation*}
Given a set $A \subseteq \mathbb{R}^{n},$ we say $A$ is convex whenever $\gamma a_{1} + (1 - \gamma) a_{2} \in A$ for all $\gamma \in[0,1],$ $a_{1}, a_{2} \in A.$
A fundamental property of convex sets is given by the separating hyperplane theorem, which states that nonempty disjoint convex sets can be separated by a hyperplane as shown below.
\begin{lemma}[{\rm Separating hyperplane theorem~\cite[Section 2.5]{Boyd2004}}]\label{SHT}
Suppose $C \subset \mathbb{R}^{n}$ and $D \subset \mathbb{R}^{n}$ are nonempty disjoint convex sets$,$ i.e.$,$ $C \cap D = \emptyset.$
Then there exist $0 \neq a \in \mathbb{R}^n$ and $ b\in \mathbb{R}$ such that $a^{T} x \leq b$ for all $x \in C$ and $ a^{T} x \geq b$ for all $x \in D.$
\end{lemma}
For an extended real-valued function $f$ on $\mathbb{R}^{n},$ we denote its domain and epigraph by $\operatorname{dom}f:= \left\{x \in \mathbb{R}^{n} \mid f(x) < +\infty \right\}$ and $\operatorname{epi}f:= \left\{(x, r) \in \mathbb{R}^{n} \times \mathbb{R} \mid f(x) \leq r\right\},$ respectively.
A function $f:\mathbb{R}^{n} \to \mathbb{R} \cup \{+\infty\}$ is said to be convex if $\operatorname{epi} f$ is a convex set in $\mathbb{R}^{n+1}.$

\subsection{Real polynomials}
Now, let us recall some notation and basic facts on SDP problems and SOS polynomials, where SDP and SOS stand for semidefinite programming and sum of squares, respectively.
Let $S^{n}$ denote the space of symmetric $n \times n$ matrices with the trace inner product and $``\succeq"$ denote the L\"owner partial order of $ S^{n},$ that is, for $M,\, N \in S^{n},\, M \succeq N$ if and only if $M-N$ is positive semidefinite.
The set consisting of all positive semidefinite $n \times n$ matrices is denoted by $S_{+}^{n}.$
For $ X, Y \in S^{n},\langle X, Y\rangle:=\operatorname{tr}(X Y),$ where ``tr" denotes the trace of a matrix.

Let $\mathbb{R}[x]$ denote the ring of real polynomials in the variable $x.$
The space of all real polynomials in the variables $x$ with degree at most $d$ is denoted by $\mathbb{R}[x]_{d}.$
The degree of a polynomial $f$ is denoted as $\operatorname{deg} f.$
We say that a real polynomial $f$ is SOS if it can be written as $f(x)=\sum_{j=1}^{l}f_{j}(x)^{2},\, x \in \mathbb{R}^{n}$ with $f_{j} \in \mathbb{R}[x].$
The cone of SOS polynomials of degree at most $2d$ is denoted by $\Sigma[x]_{2d}$.

For a multi-index $\alpha=(\alpha_1,\ldots,\alpha_n)\in\mathbb{N}^n$, let
\begin{equation*}
|\alpha|:=\sum_{i=1}^n\alpha_i, \quad \mathbb{N}_d^n:=\{\alpha\in\mathbb{N}^n\mid |\alpha|\leq d\},
\end{equation*}
and write
\begin{equation*}
x^\alpha:=x_1^{\alpha_1}\cdots x_n^{\alpha_n}.
\end{equation*}
Fix an ordering of $\mathbb{N}_d^n$ and define the standard monomial vector
\begin{equation*}
\lceil x\rceil_d:=\left(x^\alpha\right)_{\alpha\in\mathbb{N}_d^n},
\end{equation*}
whose dimension is $s(n,d):=\binom{n+d}{n}$.

For each $\alpha\in\mathbb{N}_{2d}^n$, define the symmetric matrix $\mathcal{B}_\alpha$ by
\begin{equation*}
\lceil x\rceil_d\lceil x\rceil_d^T =\sum_{\alpha\in\mathbb{N}_{2d}^n}\mathcal{B}_\alpha x^\alpha.
\end{equation*}
Given $y=(y_\alpha)_{\alpha\in\mathbb{N}_{2d}^n}$, its moment matrix of order $d$ is defined by
\begin{equation*}
\mathbf{M}_d(y):=\sum_{\alpha\in\mathbb{N}_{2d}^n}y_\alpha\mathcal{B}_\alpha.
\end{equation*}
Equivalently,
\begin{equation*}
\mathbf{M}_d(y)(\alpha,\beta)=y_{\alpha+\beta}, \quad \alpha,\beta\in\mathbb{N}_d^n.
\end{equation*}
For $y=(y_\alpha)_{\alpha\in\mathbb N_{2d}^n}$, define the linear functional
\begin{equation*}
L_y:\mathbb R[x]_{2d}\to\mathbb R, \quad L_y(p):=\sum_{\alpha\in\mathbb N_{2d}^n}p_\alpha y_\alpha,
\end{equation*}
where $p(x)=\sum_{\alpha\in\mathbb N_{2d}^n}p_\alpha x^\alpha.$
The gradient and the Hessian of a polynomial $f \in \mathbb{R}[x]$ at a point $\overline{x}$ are denoted by $\nabla f(\overline{x})$ and $\nabla^{2} f(\overline{x}),$ respectively.

Next, we recall the following result that shows how to confirm whether a polynomial can be expressed as an SOS via solving SDP problems.
\begin{proposition}{\rm \cite[Proposition 2.1]{Lasserre2015}} \label{prop2.1}
A polynomial $f \in \mathbb{R}[x]_{2d}$ has an SOS decomposition if and only if there exists $Q \in S_{+}^{s(n,d)}$ such that  $f(x)=\left\langle \lceil x \rceil_{d} \lceil x \rceil_{d}^{T} , Q\right\rangle$ for all $ x \in \mathbb{R}^{n}.$
\end{proposition}

Note that $f(x)=\sum_{\alpha \in \mathbb{N}_{2 d}^{n}} f_{\alpha} x^{\alpha}$ is an SOS polynomial if and only if we can solve the following semidefinite feasibility problem:
\begin{equation*}
\text{Find } Q \in S_{+}^{s(n,d)}  \text{ such that }  \left\langle\mathcal{B}_{\alpha}, Q\right\rangle=f_{\alpha}, \forall \alpha \in \mathbb{N}_{2d}^{n}.
\end{equation*}

At this moment, we would like to mention that not every nonnegative polynomial is an SOS polynomial; see, e.g., \cite{Reznick2000}.
Nevertheless, Lasserre and Netzer~\cite{Lasserre2007} developed SOS approximations of nonnegative polynomials based on high-degree perturbations; see also \cite{Guo2020}.
\begin{lemma}{\rm \cite[Theorems 3.1, 3.2 and Corollary   3.3]{Lasserre2007}}\label{lemma2.1}
For a given $h \in \mathbb{R}[x],$ the following statements are true.
\begin{enumerate}[\upshape (i)]
\item For any $r \geq\lceil\deg h/2\rceil,$ there exists $\varepsilon_{r}^{*} \geq 0$ such that  $h+\varepsilon\left(1+\sum_{j=1}^{n} x_{j}^{2 r}\right) $ is SOS if and only if $\varepsilon \geq \varepsilon_{r}^{*}.$
\item If $h$ is nonnegative on $[-1,1]^{n},$ then $\varepsilon_r^*$ in {\rm (i)} converges monotonically to $0$ as $r\to\infty$.
\item For any $\varepsilon>0,$ if $h$ is nonnegative on $[-1,1]^{n},$ then there exists some $r(h, \varepsilon) \in \mathbb{N}$ such that  $h+\varepsilon\left(1+\sum_{j=1}^{n} x_{j}^{2 r}\right)$ is SOS for every $r \geq r(h, \varepsilon).$
\end{enumerate}
\end{lemma}

In what follows, we recall an interesting subclass of convex polynomials introduced by Helton and Nie~\cite{Helton2010}.
\begin{definition}
{\rm
A polynomial $f\in\mathbb{R}[x]$ is called \emph{SOS-convex} if there exist an integer $p\geq 1$ and a polynomial matrix $F\in\mathbb{R}[x]^{n\times p}$ such that
\begin{equation*}
\nabla^2f(x)=F(x)F(x)^T.
\end{equation*}
}\end{definition}
The significance of SOS-convexity is that checking whether a polynomial is SOS-convex can be formulated as an SDP feasibility problem, whereas deciding whether a polynomial is convex is, in general,  NP-hard \cite{Ahmadi2013,Helton2010}.
Clearly, an SOS-convex polynomial is convex, while the converse is not true, that is, there exists a convex polynomial that is not SOS-convex \cite{Ahmadi2013}.
Besides, SOS-convex polynomials enjoy many important properties, among which, we list here the following one that is needed in our paper.

\begin{lemma}{\rm \cite{Lasserre2015}}\label{lemma2.3}
Let $ f \in \mathbb{R}[x]_{2d}$ be SOS-convex$,$ and let $y=\left(y_{\alpha}\right)_{\alpha \in \mathbb{N}_{2d}^{n}}$ satisfy  $y_{0}=1$ and $\mathbf{M}_{d}(y) \succeq 0.$
Then
\begin{equation*}
L_{y}(f) \geq f\left(L_{y}(x)\right),
\end{equation*}
where $L_{y}(x):=\left(L_{y}\left(x_{1}\right), \ldots,L_{y}\left(x_{n}\right)\right).$
\end{lemma}
We finish this section by the following lemma, which is a fundamental result that guarantees the existence of a minimizer for convex polynomial problems.

\begin{lemma}{\rm \cite{Belousov2002}}\label{lemma2.4}
Let $f, g_{1}, \ldots, g_{m}$ be convex polynomials on $\mathbb{R}^{n}.$ Let
\begin{equation*}
C:= \left\{x \in \mathbb{R}^{n} \mid g_{i}(x) \leq 0,\ i = 1, \ldots, m\right\} \neq \emptyset.
\end{equation*}
Suppose that $\inf_{x \in C} f(x) > -\infty.$
Then $\operatorname{argmin}_{x \in C} f(x) \neq \emptyset.$
\end{lemma}

\section{The Charnes--Cooper Transformation}\label{sect::3}

The Charnes--Cooper transformation (CCT) is a classical technique for reformulating fractional programs \cite{MR152370,Chen2005JOTA,schaible1974}.
As discussed in Section~\ref{sect:1}, unlike Dinkelbach-type methods, which solve a sequence of parameterized subproblems, the CCT produces a single reformulation in the lifted variables $(s,t)$ without requiring an iteratively updated ratio parameter.
In this sense, the CCT reformulation is parameter-free.
Moreover, under suitable convexity assumptions, the resulting perspective reformulation is a convex optimization problem, at the cost of introducing an additional variable and constraints arising from the transformation.

To proceed, let $K\subseteq\mathbb{R}^n$ be a nonempty set, and let $f,g:\mathbb{R}^n\to\mathbb{R}$ be functions such that $g>0$ on $K$.
Consider the following fractional program:
\begin{equation}\label{FP}
\inf_{x \in K}\ \frac{f(x)}{g(x)}. \tag{FP}
\end{equation}
The Charnes--Cooper variable transformation \cite{schaible1974}
\begin{framed}
\begin{equation}\label{CCT}
s:= \frac{x}{g(x)}, \quad t:= \frac{1}{g(x)} \tag{CCT}
\end{equation}
\end{framed}
\noindent
tells us that the problem \eqref{FP} has a very close relationship with the following problem:
\begin{equation}\label{P0}
\inf_{(s,t)\in K_{=}}  \widetilde{F} \left( s,t \right), \tag{${\rm P}_0$}
\end{equation}
where
\begin{equation*}
\widetilde{F} \left(s,t \right):= t\, f\left(s/t\right) \quad \textrm{and} \quad K_{=}:= \left\{(s,t) \in \mathbb{R}^{n} \times \mathbb{R} \mid t > 0, \ \tfrac{s}{t} \in K,\  t\,g \left(s/t\right) = 1 \right\}.
\end{equation*}

At this point, we would like to mention that, by the transformation \eqref{CCT}, there exists a one-to-one correspondence between $x\in K$ and $(s,t)\in K_{=}.$
As demonstrated in the following lemma,  the optimal solutions to the problems \eqref{FP} and \eqref{P0} correspond to each other, and have the same optimal value (cf. \cite[Lemma 2.1, Theorem 2.1 \& Remark 2.1]{Chen2005JOTA}).

\begin{lemma}\label{lemma3.1}
For the problems~\eqref{FP} and \eqref{P0}$,$ the following statements hold.
\begin{enumerate}[\upshape (i)]
\item If $x^{*}$ is an optimal solution to the problem~\eqref{FP}$,$ then the point $(s^*, t^*)$ under the transformation \eqref{CCT} is an optimal solution to the problem \eqref{P0}.
\item Conversely$,$ if $(s^*, t^*)$ is an optimal solution to the problem~\eqref{P0}$,$ then ${s^*}/{t^*}$ is an optimal solution to the problem \eqref{FP}.
\item The optimal values of the problems~\eqref{FP} and~\eqref{P0} are equal.
\end{enumerate}
\end{lemma}
In general, when $g$ is nonlinear, the feasible set $K_{=}$ of problem~\eqref{P0} may be nonconvex, even if $K$ is convex.
Thus, in addition to the problem~\eqref{P0}, we consider the following related problem~\eqref{Ptilde}:
\begin{equation}\label{Ptilde}
\inf_{(s,t)\in K_{\leq}} \widetilde{F} \left(s,t \right), \tag{$\widetilde{\rm P}$}
\end{equation}
where
\begin{equation*}
K_{\leq}:=\left\{(s,t) \in \mathbb{R}^{n} \times \mathbb{R} \mid t>0,\ \tfrac{s}{t} \in K, \  1-t\,g \left(s/t\right) \leq 0\right\}.
\end{equation*}

\begin{remark}[CCT and perspective functions]
Recall that the perspective of a function $q:\operatorname{dom}q\to\mathbb R$ is the function
\begin{equation*}
\widetilde q(z,t):=tq(z/t),
\end{equation*}
defined on
\begin{equation*}
\operatorname{dom}\widetilde q:=\left\{(z,t)\in\mathbb R^n\times(0,+\infty)\mid z/t\in\operatorname{dom}q\right\}.
\end{equation*}
The objective function
\begin{equation*}
\widetilde F(s,t)=tf(s/t)
\end{equation*}
of problems~\eqref{P0} and~\eqref{Ptilde} is precisely the perspective of $f$.
\end{remark}

\begin{lemma}\label{lemma3.2}
Assume that $f \geq 0$ on $K.$
Then
\begin{equation*}
\inf\eqref{P0} = \inf\eqref{Ptilde}.
\end{equation*}
Moreover$,$ every optimal solution to the problem~\eqref{P0} is also an optimal solution to the problem~\eqref{Ptilde}.
\end{lemma}

\begin{proof}
Note that $\inf\eqref{P0}\geq \inf\eqref{Ptilde}$ due to the fact that $K_{=}\subseteq K_{\leq},$ and observe that
\begin{equation*}
K_{\leq}=\left\{(s,t) \in \mathbb{R}^{n} \times \mathbb{R} \mid \exists\, (\bar s,\bar t)\in K_{=}, \ \tau\geq 1 \textrm{ such that } (s,t)=\tau(\bar s,\bar t)\right\}.
\end{equation*}
Indeed, for any $(s,t)\in\mathcal K_{\leq}$, let
\begin{equation*}
x:=\frac{s}{t},\quad \bar t:=\frac{1}{g(x)},\quad \bar s:=\bar t x,\quad \tau:=t g(x).
\end{equation*}
Then $\tau\geq1$, $(\bar s,\bar t)\in\mathcal K_=$, and $(s,t)=\tau(\bar s,\bar t)$.
The converse implication is immediate.

If $(s,t)=\tau(\bar s,\bar t)$ with $\tau\geq 1$ and $(\bar s,\bar t)\in K_{=}$, then
\begin{equation}\label{lemma3.2rel1}
\widetilde{F}(s,t)=t\, f\left({s}/{t}\right)=\tau\bar t\, f\left({\bar s}/{\bar t}\right)=\tau \widetilde{F}(\bar s,\bar t) \geq \widetilde{F}(\bar s,\bar t),
\end{equation}
where the last inequality follows from $\widetilde{F}(\bar s,\bar t)\geq 0$ (since $f\geq 0$ on $K$).
Therefore, we have $\inf\eqref{P0}\leq \inf\eqref{Ptilde},$ and hence $\inf\eqref{P0}= \inf\eqref{Ptilde}.$

Now, let $(s^*,t^*)\in K_=$ be an optimal solution to the problem~\eqref{P0}.
Suppose to the contrary that $(s^*,t^*)$ is not optimal for \eqref{Ptilde}.
Then there exists $(s,t)\in K_{\leq}$ such that
\begin{equation*}
\widetilde{F}(s,t)<\widetilde{F}(s^*,t^*).
\end{equation*}
By the representation of $K_{\leq}$ given above, there exist $(\bar s,\bar t)\in K_{=}$ and $\tau\geq 1$ such that
\begin{equation*}
(s,t)=\tau(\bar s,\bar t).
\end{equation*}
It follows from~\eqref{lemma3.2rel1} that
\begin{equation*}
\widetilde{F}(\bar s,\bar t)\leq\widetilde{F}(s,t)<\widetilde{F}(s^*,t^*),
\end{equation*}
which contradicts the optimality of $(s^*,t^*)$.
Hence, $(s^*,t^*)$ is also an optimal solution to the problem~\eqref{Ptilde}.
\end{proof}

\begin{remark}{\rm
The conclusion of Lemma~\ref{lemma3.2} cannot be strengthened to the result that \eqref{P0} and \eqref{Ptilde} admit the same optimal solutions under the assumption $f\geq 0$ on $K$.
Indeed$,$ consider $f(x)=x^2,$ $g(x)\equiv 1,$ and $K=[-1,1].$
Then \eqref{P0} has the unique optimal solution $(s,t)=(0,1)$, whereas every point $(0,t)$ with $t\geq 1$ is an optimal solution to the problem~\eqref{Ptilde}.
Thus$,$ \eqref{P0} and \eqref{Ptilde} have the same optimal value$,$ but their optimal solution sets are different.

On the other hand, if $f>0$ on $K$, then the optimal solution sets of \eqref{P0} and \eqref{Ptilde} coincide.
Indeed$,$ suppose that $(s,t) \in K_{\leq}\setminus K_{=}$ is an optimal solution to the problem~\eqref{Ptilde}.
Then there exist $(\bar s,\bar t)\in K_{=}$ and $\tau>1$ such that $(s,t)=\tau(\bar s,\bar t)$, which yields $\widetilde{F}(s,t)=\tau\widetilde{F}(\bar s,\bar t)>\widetilde{F}(\bar s,\bar t)$, contradicting the optimality of $(s,t)$.
Therefore, every optimal solution to problem~\eqref{Ptilde} belongs to $\mathcal K_=$.
Since $\inf\eqref{P0}=\inf\eqref{Ptilde}$, every such point is also an optimal solution to problem~\eqref{P0}.
Together with Lemma~\ref{lemma3.2}, this shows that the two problems have the same optimal solution set.
}\end{remark}

The following proposition shows that any optimal solution to the problem~\eqref{Ptilde} can be used to recover an optimal solution to the problem~\eqref{FP}.
\begin{proposition}\label{prop3.1}
Assume that $f\geq 0$ on $K$, and let $(s^*,t^*)$ be an optimal solution to problem~\eqref{Ptilde}.
Then
\begin{equation*}
x^*:=\frac{s^*}{t^*}
\end{equation*}
is an optimal solution to problem~\eqref{FP}.

Moreover, define \begin{equation*}
\bar t:=\frac{1}{g(x^*)}, \quad \bar s:=\bar t x^*.
\end{equation*}
Then $(\bar s,\bar t)\in K_{=}$ is an optimal solution to problem~\eqref{P0}.
\end{proposition}

\begin{proof}
Let
\begin{equation*}
x^*:=\frac{s^*}{t^*}, \quad \bar t:=\frac{1}{g(x^*)}, \quad \bar s:=\bar t x^*.
\end{equation*}
Since $(s^*,t^*)\in K_{\leq}$, we have $x^*\in K$ and
\begin{equation*}
1-t^*g(x^*)\leq0.
\end{equation*}
Since $g(x^*)>0$, it follows that
\begin{equation*}
t^*\geq\frac{1}{g(x^*)}=\bar t.
\end{equation*}
Hence, using $f(x^*)\geq 0,$
\begin{equation*}
\widetilde F(s^*,t^*)=t^*f(x^*)\geq\frac{f(x^*)}{g(x^*)}.
\end{equation*}
On the other hand, by Lemma~\ref{lemma3.1} and Lemma~\ref{lemma3.2}, we see that
\begin{equation*}
\widetilde F(s^*,t^*)=\inf\eqref{Ptilde}=\inf\eqref{P0}=\inf\eqref{FP}.
\end{equation*}
Thus, we have
\begin{equation*}
\inf\eqref{FP}=\widetilde F(s^*,t^*)\geq\frac{f(x^*)}{g(x^*)}.
\end{equation*}
Since $x^*\in K,$ we also have
\begin{equation*}
\frac{f(x^*)}{g(x^*)}\geq \inf\eqref{FP}.
\end{equation*}
Combining the above inequalities yields
\begin{equation*}
\frac{f(x^*)}{g(x^*)}=\inf\eqref{FP},
\end{equation*}
and thus, $x^*$ is an optimal solution to problem~\eqref{FP}.

Now define
\begin{equation*}
\overline t:=\frac{1}{g(x^*)}, \quad \overline s:=\overline t\,x^*.
\end{equation*}
Then $(\overline s,\overline t)\in K_=,$ and
\begin{equation*}
\widetilde F(\overline s,\overline t)=\overline t\,f(x^*)=\frac{f(x^*)}{g(x^*)}=\inf\eqref{FP}=\inf\eqref{P0}.
\end{equation*}
Hence, $(\overline s,\overline t)$ is an optimal solution to the problem~\eqref{P0}.
\end{proof}

The following lemma shows that the problem~\eqref{Ptilde} is a convex optimization problem under certain conditions.
\begin{lemma}\label{lemma3.3}
If $K$ is convex and $f$ and $-g$ are convex functions$,$ then problem~\eqref{Ptilde} is a convex optimization problem.
\end{lemma}

\begin{proof}
Let $(s_1,t_1),$ $(s_2,t_2)\in K_{\leq}$ and let $\theta\in[0,1].$
Set
\begin{equation*}
s:=\theta s_{1}+(1-\theta) s_{2}, \quad t:=\theta t_{1}+(1-\theta) t_{2}.
\end{equation*}
Since $t_{1}, t_{2}>0,$ we have $t > 0.$
Moreover, we see that
\begin{equation*}
\frac{s}{t}=\gamma \frac{s_{1}}{t_{1}}+(1-\gamma) \frac{s_{2}}{t_{2}},
\end{equation*}
where
\begin{equation*}
\gamma:=\frac{\theta t_1}{t}, \quad 1-\gamma=\frac{(1-\theta)t_2}{t}.
\end{equation*}
Since $t=\theta t_1+(1-\theta)t_2>0$, it follows that $\gamma\in[0,1]$.
By the convexity of $ K,$ we have $\frac{s}{t} \in K .$
Using the concavity of $g,$ we obtain
\begin{equation*}
g \left(s/t\right) \geq \gamma g\left(s_{1}/t_{1}\right)+(1-\gamma) g \left(s_{2}/t_{2}\right).
\end{equation*}
Multiplying both sides by $t$ yields
\begin{equation*}
t\,g \left(s/t\right) \geq \theta t_{1} \cdot g \left(s_{1}/t_{1}\right)+(1-\theta) t_{2} \cdot g \left(s_{2}/t_{2}\right) \geq 1,
\end{equation*}
where the last inequality follows from $\left(s_{i}, t_{i}\right) \in K_{\leq}$ for $i=1,2.$
Therefore, $ (s,t) \in K_{\leq},$ which establishes the convexity of $K_{\leq}.$
By definition, we have
\begin{equation*}
\operatorname{dom}\widetilde{F}=\left \{(s,t) \mid s/t \in \operatorname{dom} f, \ t>0\right\},
\end{equation*}
and $\widetilde F(s,t)=t\,f(s/t)$ is indeed the perspective of $f$.
Since the perspective operation preserves convexity (see \cite[page 89]{Boyd2004}), we know that $\widetilde F$ is a convex function on its domain.
Therefore, \eqref{Ptilde} is a convex optimization problem.
\end{proof}

\section{Solving Fractional Programs with Convex Polynomials via the CCT Method}\label{sect::4}

We now consider problem~\eqref{FP} under the following polynomial structure.
Let
\begin{equation}\label{K}
K:=\left\{x\in\mathbb{R}^n\mid h_i(x)\leq0,\quad i=1,\ldots,m\right\},
\end{equation}
where $f$, $-g$, and $h_i$, $i=1,\ldots,m$, are convex polynomials of degree at most $2d$.
Accordingly, problem~\eqref{FP} takes the form
\begin{equation*}
\inf_{x\in\mathbb{R}^n}\left\{\frac{f(x)}{g(x)}\mid h_i(x)\leq0,\quad i=1,\ldots,m\right\}.
\end{equation*}

Throughout this section, we impose the following commonly used assumptions:
\begin{itemize}
\item[\textbf{(A1)}] The set $K$ in~\eqref{K} is nonempty.
\item[\textbf{(A2)}] $f(x)\geq0$ and $g(x)>0$ for every $x\in K$.
\end{itemize}

Under this representation of $K$, problem~\eqref{Ptilde} can be written explicitly as
\begin{align}\label{PCCT}
\inf_{(s,t)\in \mathbb{R}^{n}\times \mathbb{R}} \quad &t\,f (s/t) \tag{P$_{\rm CCT}$} \\
{\rm s.t.}\qquad \ & t\,h_{i} (s/t) \le 0, \ i=1,\ldots m, \nonumber\\
& 1- t\, g(s/t) \le 0, \nonumber\\
& t>0. \notag
\end{align}
Moreover, Lemma~\ref{lemma3.3} shows that the problem \eqref{PCCT} is a convex optimization problem since $K$ is convex and $f$ and $-g$ are convex.

The following theorem establishes existence of optimal solutions to the problem \eqref{FP} as well as its CCT reformulation problem \eqref{PCCT}.

\begin{theorem}\label{Thm4.1}
Let $f,-g,h_1,\ldots,h_m:\mathbb{R}^n\to\mathbb{R}$ be convex polynomials$,$ and let assumptions {\bf (A1)} and {\bf (A2)} hold and $\sup_{x \in K} \{g(x)\}< +\infty.$ Then the following statements hold$:$
\begin{enumerate}[\upshape (i)]
\item The problem~\eqref{FP} attains its optimal value at some $x^*\in K$.
\item The problem~\eqref{PCCT} attains its minimum at some $(s^*,t^*)$ with $t^*>0$ and $t^* g(s^*/t^*)-1 = 0,$ and the optimal values of the problems \eqref{FP} and \eqref{PCCT} coincide.
\end{enumerate}
\end{theorem}

\begin{proof}
(i) Let $r^*:=\inf\eqref{FP},$ which is finite by assumption {\bf (A2)}.
Define $\Phi_{r}(x):=f(x)-r\,g(x)$.
Then, for every $x\in K$,
\begin{equation*}
\Phi_{r^*}(x)=f(x)-r^* g(x)=g(x)\left(\tfrac{f(x)}{g(x)}-r^*\right) \geq 0,
\end{equation*}
and hence $\inf_{x\in K}\Phi_{r^*}(x)\geq 0$.
Moreover, let $M:=\sup_{x\in K}\{g(x)\}<\infty$ and taking a sequence $x_k\in K$ with $\tfrac{f(x_k)}{g(x_k)}\to r^*$ as $k\to\infty$ gives
\begin{equation*}
0\leq \Phi_{r^*}(x_k)=g(x_k) \left(\frac{f(x_k)}{g(x_k)}-r^*\right)\leq M\left(\frac{f(x_k)}{g(x_k)}-r^*\right) \to 0
\end{equation*}
as $k\to\infty,$ and so, we have
\begin{equation*}
\inf_{x\in K}\Phi_{r^*}=0.
\end{equation*}
Since $\Phi_{r^*}$ is a convex polynomial and $K$ is a convex basic semialgebraic set defined by convex polynomials, by Lemma~\ref{lemma2.4}, the optimal value of $\Phi_{r^*}$ over $K$ is attained.
Hence there exists $x^*\in K$ such that $\Phi_{r^*}(x^*)=0,$ and so, we have
\begin{equation*}
\frac{f(x^*)}{g(x^*)}=r^*.
\end{equation*}
Thus, \eqref{FP} attains its optimal value at $x^*$.
	
(ii) Let $x^*$ be the optimal solution obtained in part~{\rm (i)}, and let
\begin{equation*}
t^*:=\frac{1}{g(x^*)}, \quad s^*:=t^*x^*.
\end{equation*}
By Lemma~\ref{lemma3.1}, $(s^*,t^*)$ is an optimal solution to problem~\eqref{P0}, and problems~\eqref{FP} and~\eqref{P0} have the same optimal value.
Since $f\geq0$ on $K$, Lemma~\ref{lemma3.2} shows that $(s^*,t^*)$ is also an optimal solution to problem~\eqref{PCCT}.
Moreover,
\begin{equation*}
t^*g(s^*/t^*)=1.
\end{equation*}
\end{proof}

\begin{remark}{\rm
\begin{enumerate}[\upshape (i)]
\item The assumption $\sup _{x \in K}\{g(x)\}<+\infty$ in Theorem \ref{Thm4.1} is crucial for ensuring that the optimal values of the problems \eqref{FP} and \eqref{PCCT} are attained.
If this assumption is removed, the conclusion of Theorem \ref{Thm4.1} may fail, as demonstrated by Example~\ref{example-a}.
However, this assumption is not very restrictive; it simply requires $g$ to be bounded above on $K$, or equivalently, $-g$ to be bounded below on $K$.
In particular, since $g$ is continuous, this condition is automatically satisfied whenever $K$ is compact.
			
\item It is worth mentioning that$,$ even with this boundedness assumption, existence results may fail beyond the setting of Theorem~\ref{Thm4.1}.
For instance, attainment can fail when $f$ is convex but nonpolynomial or when $f$ is polynomial but nonconvex.
\begin{itemize}
\item {\bf convex but nonpolynomial case.}
Let $n=1,$ $K=\mathbb{R},$ $g(x)\equiv 1,$ and $f(x)=e^{-x}$ $($convex, but not a polynomial$).$
Then $g>0$ on $K$ and $\sup_{x\in K}\{g(x)\}=1<\infty$, but
\begin{equation*}
\inf_{x\in K}\frac{f(x)}{g(x)}=\inf_{x\in\mathbb{R}} e^{-x}=0,
\end{equation*}
the infimum is not attained.
\item {\bf polynomial but nonconvex case.}
Let $n=2$, $K=\mathbb{R}^2,$ $g(x_1,x_2)\equiv 1,$ and $f(x_1,x_2)=(x_1x_2-1)^2+x_1^2$ $($a polynomial, but not convex$).$
Then $g>0$ on $K$ and $\sup_{x\in K}\{g(x)\}=1<\infty$, but
\begin{equation*}
\inf_{(x_1,x_2)\in K}\frac{f(x_1,x_2)}{g(x_1,x_2)}=\inf_{(x_1,x_2)\in\mathbb{R}^2} \{(x_1x_2-1)^2+x_1^2\} = 0,
\end{equation*}
the infimum is not attained.
\end{itemize}
Consequently, attainment may fail when the convex polynomial assumptions of Theorem~\ref{Thm4.1} are relaxed, even though $\sup_{x\in K}g(x)<+\infty$.
\end{enumerate}
}\end{remark}

\begin{example}\label{example-a}{\rm
Let $n:=1,$ and consider the following simple fractional problem with one inequality constraint,
\begin{align}\label{exam0}
\inf_{x \in \mathbb{R}} \left\{ \frac{f(x)}{g(x)} \mid h_1(x) \leq 0 \right\}, \tag{{\rm FP$_0$}}
\end{align}
where $f(x) = 1,$ $g(x) = -x$ and $h_1(x) = x + 1.$
Observe that the feasible set is
\begin{equation*}
K = \{ x \in \mathbb{R} \mid h_1(x) \leq 0 \} = (-\infty,-1].
\end{equation*}
Clearly, assumptions {\bf (A1)} and {\bf (A2)} are satisfied.
However,
\begin{equation*}
\sup_{x \in K} \{g(x)\} = \sup_{x \leq -1} \{-x\} = +\infty.
\end{equation*}
A simple calculation shows that the infimum of the problem~\eqref{exam0} is $0,$ but this value is never attained for any $x \in K.$
Therefore, the problem~\eqref{exam0} does not attain its optimal value.
		
Next, we consider the CCT reformulation \eqref{examd0} with its objective function $t f(s/t) = t \cdot 1 = t$,
\begin{align}\label{examd0}
\inf_{(s,t) \in \mathbb{R}^{n} \times \mathbb{R}} \left\{ t \mid t h_1(s/t) \leq 0,\ 1-t\,g(s/t) \leq 0,\ t > 0 \right\}, \tag{${\rm P_{\text{CCT0}}}$}
\end{align}
Thus, the feasible set is $\{ (s,t) \mid s+t\leq 0,\ s \leq -1,\ t > 0 \}$.
Clearly, the infimum is $0$, but no feasible $(s,t)$ achieves $t=0$.
Therefore, the problem~\eqref{examd0} does not attain its minimum.
		
Overall, both problems \eqref{exam0} and \eqref{examd0} fail to attain their optimal values in this example, demonstrating that the boundedness assumption cannot be removed from Theorem~\ref{Thm4.1}.
}\end{example}

Now, we derive the Lagrange dual of the problem~\eqref{PCCT}.
To do this, we define the Lagrangian function of the problem~\eqref{PCCT} as
\begin{equation*}
\mathcal{L}(s,t;\lambda,\gamma):= t\,f\left(\frac{s}{t}\right)+\sum_{i=1}^m \lambda_i\left(t\,h_i\left(\frac{s}{t}\right)\right)+\gamma\left(1-t\,g\left(\frac{s}{t}\right)\right),
\end{equation*}
with multipliers $\lambda_i\geq 0$ and $\gamma\geq 0$ associated with the inequality constraints.
Substituting $x:=s/t$ and using $t>0$, we obtain
\begin{equation*}
\inf_{t>0,\,x\in\mathbb{R}^n}\ \mathcal{L}(s,t;\lambda,\gamma)
= \inf_{t>0,\,x\in\mathbb{R}^n} \left\{\, t\left(f(x)+\sum_{i=1}^m\lambda_i h_i(x)-\gamma g(x)\right) + \gamma \,\right\}.
\end{equation*}
This infimum is finite if and only if
\begin{equation*}
f(x)+\sum_{i=1}^m\lambda_i h_i(x)-\gamma g(x)\ \geq 0
\end{equation*}
for all $x\in\mathbb{R}^n,$ in which case the infimum equals $\gamma$, approached as $t\downarrow0$.
Hence the Lagrange dual problem reads
\begin{align}\label{LD}
\sup_{\lambda,\,\gamma} &\quad \gamma \tag{LD} \\
{\rm s.t.} & \quad  f(x)+\sum_{i=1}^{m}\lambda_{i}h_{i}(x)-\gamma g (x)\geq0, \ \forall x \in \mathbb{R}^n, \nonumber\\
& \quad \lambda_{i}\ge0,\ i=1,\ldots,m, \ \gamma \ge0. \notag
\end{align}
By the standard weak duality theorem (see, e.g., \cite{Boyd2004}), the dual objective provides a lower bound on the primal optimal value, i.e., $\sup\eqref{LD}\leq \inf\eqref{PCCT}$.
In what follows, we present a strong duality result.

%
%

%

\begin{theorem}\label{Thm4.2}
If the problem~\eqref{FP} satisfies the Slater condition$,$ i.e.$,$ there exists a point $\hat{x}$ such that $ h_{i}({\hat{x}})<0 $ for all $ i=1, \ldots, m,$  then
\begin{equation*}
	\inf \eqref{PCCT} = \max \eqref{LD}.
\end{equation*}
\end{theorem}
\begin{proof}
We first show that the problem~\eqref{PCCT} has a strictly feasible solution, say  $(\hat{s}, \hat{t}) \in \mathbb{R}^{n} \times (0,+\infty)$, i.e.,
\begin{align*}
& \hat{t} h_{i} (\hat{s}/\hat{t}) < 0,\, i=1, \ldots m, \\
& 1-\hat{t}\, g (\hat{s}/\hat{t})< 0 , \\
&\hat{t}>0, \frac{\hat{s}}{\hat{t}} \in \mathbb{R}^{n}.
\end{align*}
Since the Slater condition for the problem~\eqref{FP} holds, there exists $\hat{x}$ such that $ h_{i}({\hat{x}})<0 $ for all $ i=1, \ldots, m$.
For any $\varepsilon>0,$ let $\hat{t}=\frac{1+\varepsilon}{g(\hat{x})},\,\hat{s}=\hat{t}\hat{x},$ then together with the fact that $g(\hat{x})>0,$ we have
\begin{align*}
\hat{t} h_{i} (\hat{s}/\hat{t}) & =\frac{1+\varepsilon}{g(\hat{x})} h_{i}({\hat{x}})<0,\\
1-\hat{t}\, g (\hat{s}/\hat{t}) & =1- \frac{1+\varepsilon}{g(\hat{x})}g(\hat{x})<0.
\end{align*}
Thus, the point $(\hat{s},\hat{t})$ is a strictly feasible solution for the problem~\eqref{PCCT}, in other words, the problem~\eqref{PCCT} satisfies the Slater condition.

Next, let $p^{*}:=\inf \eqref{PCCT},$ and
consider the following sets
\begin{align*}
\mathcal{A} =\Bigl\{(u,v,p)\in \mathbb{R}^m \times \mathbb{R} \times \mathbb{R} \mid \, & \exists \, (s,t) \in \mathbb{R}^n \times (0,+\infty) \textrm{ s.t. } t\, h_{i}({s}/{t})\leq u_{i},\ i=1,\ldots,m,\\
& 1-t\,g({s}/{t})\leq v, \ t\, f({s}/{t})\leq p \Bigr\} ,\\
\mathcal{B} =\{(0,0, w) \in \mathbb{R}^m \times \mathbb{R} \times \mathbb{R} \mid \,& w<p^{*}\} .
\end{align*}
Observe that $\mathcal{A} \cap \mathcal{B}= \emptyset.$
Indeed, suppose $(u, v, p) \in \mathcal{A} \cap \mathcal{B}.$
Since $(u, v, p) \in \mathcal{B},$ we have $u=0,$ $v=0,$ and $p<p^{*},$ which is impossible since $p^{*}$ is the optimal value of the  problem \eqref{PCCT}.

Since $\mathcal{A}$ and $\mathcal{B}$ are convex and disjoint, by the hyperplane separation theorem (see Lemma~\ref{SHT}), there exist $0 \neq (\lambda, \gamma,\nu)\in \mathbb{R}^m \times \mathbb{R} \times \mathbb{R}$ and  $\alpha \in \mathbb{R}$ such that
\begin{eqnarray}
\lambda^{T} u+\gamma v+\nu p\geq \alpha , && \forall (u, v, p) \in \mathcal{A}, \label{4} \\
\lambda^{T} u+\gamma v+\nu p \leq \alpha ,&&\forall (u, v, p) \in \mathcal{B}. \label{5}
\end{eqnarray}
Then $\lambda\geq0,$ $\gamma\geq 0,$ and $\nu\geq 0.$
Indeed, suppose to the contrary that $\lambda_j<0$ for some $j$.
From \eqref{4}, we have $\lambda^{\top}u+\gamma v+\nu p\to -\infty$ as $u_j\to +\infty,$ contradicting \eqref{4}.
Thus $\lambda_i\geq 0$ for all $i$.
The same argument with $v$ and $p$ yields $\gamma\geq 0$ and $\nu\geq 0$.
From~\eqref{5}, we have $\nu p \leq \alpha $ for all $ p<p^{*} ,$ and hence, taking the supremum over $p<p^*$ gives $\nu p^*\leq \alpha$.
Under the substitution $x={s}/{t},$ for any $(x,t)$, let $(u_{i},v,p)=(t\, h_{i}(x),1-t\, g(x),t\, f(x)),$ by \eqref{4} we have
\begin{equation}\label{6}
\sum_{i=1}^{m} \lambda_{i} t\, h_{i}(x)+\gamma (1-t\, g(x))+\nu t\, f(x) \geq \alpha \geq \nu p^{*} .
\end{equation}
Assume that $\nu>0.$
Dividing \eqref{6} by $\nu$ yields
\begin{equation}\label{6-1}
t\left(f(x)+\sum_{i=1}^{m}\bar{\lambda}_{i}h_{i}(x)-\bar{\gamma} g(x)\right)+\bar{\gamma} \geq p^{*} ,
\end{equation}
where $\bar{\lambda}_{i}:=\frac{\lambda_{i}}{\nu},$ $i=1,\ldots,m,$ and $\bar{\gamma}:=\frac{\gamma}{\nu}.$
Then we have
\begin{equation*}
f(x)+\sum_{i=1}^{m}\bar{\lambda}_{i}h_{i}(x)-\bar{\gamma} g(x)\geq 0
\end{equation*}
for all $x\in \mathbb{R}^{n}.$
Indeed, if $f(x)+\sum_{i=1}^{m}\bar{\lambda}_{i}h_{i}(x)-\bar{\gamma} g(x)<0$ for some $x\in\mathbb{R}^n,$ then letting $t\to\infty$ in \eqref{6-1}, we obtain
$t\left(f(x)+\sum_{i=1}^{m}\bar{\lambda}_{i}h_{i}(x)-\bar{\gamma} g(x)\right)+\bar{\gamma}\to-\infty,$ contradicting~\eqref{6-1}.
So, $(\bar\lambda,\bar\gamma)$ is feasible for the problem~\eqref{LD}.
Furthermore, letting $t\to 0^+$ in \eqref{6-1} yields $\bar\gamma\geq p^*$.
Consequently, it follows that $\max \eqref{LD} \geq \inf \eqref{PCCT}.$
By weak duality we have  $\sup \eqref{LD} \leq \inf \eqref{PCCT},$ and thus $\max \eqref{LD} = \inf \eqref{PCCT}.$
This shows that strong duality holds and that the dual optimum is attained in the case $\nu>0.$

Suppose now that $\nu=0$.
Then~\eqref{6} gives
\begin{equation}\label{6-2}
t\left(\sum_{i=1}^m\lambda_i h_i(x)-\gamma g(x)\right)+\gamma\geq0 \quad \forall x\in\mathbb{R}^n,\ t>0.
\end{equation}
Let $\hat{x}$ be a Slater point for problem~\eqref{FP}.
Then
\begin{equation*}
h_i(\hat{x})<0,\quad i=1,\ldots,m, \quad g(\hat{x})>0.
\end{equation*}
Since $\lambda\geq0$ and $\gamma\geq0$, if $(\lambda,\gamma)\neq(0,0)$, then
\begin{equation*}
\sum_{i=1}^m\lambda_i h_i(\hat{x})-\gamma g(\hat{x})<0.
\end{equation*}
Putting $x=\hat{x}$ in~\eqref{6-2} and letting $t\to+\infty$ yields a contradiction.
Hence $(\lambda,\gamma)=(0,0)$, which contradicts $(\lambda,\gamma,\nu)\neq 0$.
Therefore, $\nu>0$.
\end{proof}

Next, we construct an SOS relaxation of problem~\eqref{LD} using the high-degree perturbation method of Lasserre and Netzer~\cite{Lasserre2007}.
To this end, we impose the following assumption.
\begin{itemize}
\item[\textbf{(A3)}] The feasible set $K$ in~\eqref{K} is bounded.
\end{itemize}
Since $K$ is closed and bounded, assumption {\bf (A3)} implies that $K$ is compact.
Moreover, since $g>0$ on $K$, we have
\begin{equation*}
g_{\min}:=\min_{x\in K}g(x)>0.
\end{equation*}
Choose
\begin{equation*}
L>\max_{x\in K}\|x\|\quad\text{and}\quad C>\frac{2}{g_{\min}}.
\end{equation*}
For each integer $r\geq d$, define
\begin{equation*}
\Theta_r(x):=\sum_{j=1}^n\left(\frac{x_j}{L}\right)^{2r}.
\end{equation*}
Then $\Theta_r(x)<1$ for all $x\in K$.

For each integer $r\geq d$, consider the following SOS relaxation:
\begin{align}\label{D_r}
\sup_{\lambda,\gamma,\eta}\quad& \gamma-C\eta \tag{${\rm D}_r$}\\
{\rm s.t.}\quad& f+\sum_{i=1}^{m}\lambda_i h_i-\gamma g+\eta(1+\Theta_r) \in\Sigma[x]_{2r},\nonumber\\
&\lambda_i\geq0,\ i=1,\ldots,m,\ \gamma\geq0,\ \eta\geq0.\notag
\end{align}
By Proposition~\ref{prop2.1}, the SOS constraint in~\eqref{D_r} admits a Gram-matrix representation.
Hence, problem~\eqref{D_r} is an SDP.

The following theorem shows that the sequence of optimal values associated with problems~\eqref{D_r} converges asymptotically to the global minimum of problem~\eqref{FP}.
\begin{theorem}\label{Thm4.3}
Let $f,-g,h_1,\ldots,h_m:\mathbb{R}^n\to\mathbb{R}$ be convex polynomials.
Suppose that assumptions {\bf (A1)-(A3)} hold.
If the problem \eqref{FP} satisfies the Slater condition, then
\begin{equation*}
\lim_{r \rightarrow \infty} \sup \eqref{D_r}= \max \eqref{LD}= \inf \eqref{FP}.
\end{equation*}
\end{theorem}
\begin{proof}
First, we prove that
\begin{equation*}
\sup\eqref{D_r}\leq\inf\eqref{FP}
\end{equation*}
for every $r\geq d$.
Let $(\lambda, \gamma, \eta)$ be a feasible solution to the problem~\eqref{D_r}.
Then
\begin{equation*}
f(x)+\sum_{i=1}^{m} \lambda_{i} h_{i}(x)-\gamma g(x)+\eta\left(1+\Theta_{r}(x)\right) \geq 0 ,\, \forall x \in \mathbb{R}^{n} .
\end{equation*}
For any $x \in K ,$ since $\lambda_{i} \geq 0 $ and  $h_{i}(x) \leq 0 ,$ we obtain
\begin{equation*}
f(x)-\gamma g(x)+\eta\left(1+\Theta_{r}(x)\right) \geq 0, \, \forall x \in K
\end{equation*}
which gives
\begin{equation*}
\frac{f(x)}{g(x)} \geq \gamma-\eta \cdot \frac{1+\Theta_{r}(x)}{g(x)}  \geq \gamma-\eta \cdot \frac{2}{\min_{x \in K}\{g(x)\}} \geq \gamma-C\eta,\,  \forall x \in K.
\end{equation*}
Taking the infimum over $x \in K$ gives $\gamma-C \eta \leq \inf \eqref{FP} .$
Since $(\lambda, \gamma, \eta)$ is arbitrary, we conclude $\sup \eqref{D_r} \leq \inf \eqref{FP} .$

Next, we prove that $\lim_{r \rightarrow \infty} \sup \eqref{D_r}= \max \eqref{LD}.$
Suppose that $(\lambda^{*}, \gamma^{*})$ is an optimal solution to the problem \eqref{LD}.
Then
\begin{equation*}
\Phi(x):=f(x)+\sum_{i=1}^{m} \lambda_{i}^{*} h_{i}(x)-\gamma^{*} g(x) \geq 0, \quad \forall x \in \mathbb{R}^{n},
\end{equation*}
and $\Phi(x)$ is a convex polynomial.
Define $\tilde{\Phi}(z):=\Phi(Lz)$ for $z \in \mathbb{R}^{n}.$
Then $\tilde{\Phi}$ is a convex polynomial satisfying  $\tilde{\Phi}(z) \geq 0$ for all $z \in \mathbb{R}^{n} ,$ and in particular, $\tilde{\Phi}(z) \geq 0$ for all  $z \in[-1,1]^{n} .$
By Lemma \ref{lemma2.1} there exist  $\varepsilon_{r}^{*}\geq0$ with $\varepsilon_{r}^{*} \searrow 0 $ as $r \rightarrow \infty$ such that
\begin{equation*}
\tilde{\Phi}(z)+\varepsilon_{r}^{*}\left(1+\sum_{j=1}^{n} z_{j}^{2r}\right) \in \Sigma[z]_{2r}.
\end{equation*}
Substituting $z=x/L$ yields
\begin{equation*}
\Phi(x)+\varepsilon_{r}^{*}\left(1+\sum_{j=1}^{n}\left(x_{j} / L\right)^{2r}\right) \in \Sigma[x]_{2r},
\end{equation*}
that is,
\begin{equation*}
f(x)+\sum_{i=1}^{m} \lambda_{i}^{*} h_{i}(x)-\gamma^{*}g(x)+\varepsilon_{r}^{*}\left(1+\Theta_{r}(x)\right) \in \Sigma[x]_{2r} .
\end{equation*}
Hence $(\lambda^{*},\gamma^{*},\varepsilon_r^{*})$ is feasible for problem~\eqref{D_r} with objective value $\gamma^{*}-C\varepsilon_r^{*}$.
Therefore,
\begin{equation*}
\gamma^{*}-C\varepsilon_r^{*}\leq \sup\eqref{D_r}\leq \gamma^{*}.
\end{equation*}
Since $\varepsilon_r^{*}\to0$ as $r\to\infty$, we obtain
\begin{equation*}
\lim_{r\to\infty}\sup\eqref{D_r}=\gamma^{*}.
\end{equation*}
The desired conclusion now follows from Theorem~\ref{Thm4.2}.
\end{proof}

At this point, we would like to mention that, even though Theorem \ref{Thm4.3} establishes the asymptotic convergence, the optimal value of the problem~\eqref{D_r} only provides a lower bound and does not recover an optimal solution to the problem~\eqref{FP} directly.
In order to extract optimal solutions to the primal problem~\eqref{FP}, we now consider the following Lagrangian dual of the problem~\eqref{D_r},
\begin{align}\label{Q_r}
\inf_{y\in\mathbb{R}^{s(n,2r)}}\quad
& L_y(f)  \tag{${\rm Q}_r$}\\
\text{s.t.}\hspace{8mm}
& L_y(h_i)  \le 0,\quad i=1,\ldots,m, \nonumber \\
& L_y(1+\Theta_r) \le C, \nonumber \\
& 1 - L_y(g) \le 0,\nonumber \\
&\mathbf{M}_r(y)\succeq 0. \nonumber
\end{align}

\begin{theorem}\label{Thm4.4}
Suppose that assumptions {\bf (A1)-(A3)} hold.
If the problem~\eqref{FP} satisfies the Slater condition$,$ then for every integer $r\geq d$,
\begin{equation*}
\max\eqref{D_r}=\inf\eqref{Q_r}.
\end{equation*}
Furthermore$,$ suppose that $\bar{y}$ is an optimal solution to the problem \eqref{Q_r} with $\bar{y}_{0} \neq 0 .$ If
\begin{equation*}
\operatorname{rank} \mathbf{M}_r(\bar y)=1,
\end{equation*}
then
\begin{equation*}
\bar{x}:=\frac{1}{\bar{y}_{0}}\left(L_{\bar{y}}\left(x_{1}\right),\ldots,L_{\bar{y}}\left(x_{n}\right)\right)
\end{equation*}
is an optimal solution to the problem~\eqref{FP}.
\end{theorem}

\begin{proof}
It suffices to show that the problem~\eqref{Q_r} admits a strictly feasible solution.
Since the problem~\eqref{FP} satisfies the Slater condition, by the proof of Theorem~\ref{Thm4.2}, the problem~\eqref{PCCT} admits a strictly feasible solution.
That is, there exists $(\hat{s},\hat{t})\in \mathbb{R}^{n}\times (0,+\infty)$ such that $\hat{t}\,h_{i}(\frac{\hat{s}}{\hat{t}})<0,$ $i=1,\ldots,m,$ and $1-\hat{t}\,g(\frac{\hat{s}}{\hat{t}})<0.$
Denote $\hat{x}:=\hat{s}/\hat{t}$.
In the proof of Theorem~\ref{Thm4.2}, we may choose
\begin{equation*}
\hat{t}=\frac{1+\varepsilon_{0}}{g(\hat{x})}
\end{equation*}
for any sufficiently small $\varepsilon_{0}>0.$
Moreover, due to the continuity of the functions $h_{i},$ $i=1,\ldots,m,$ $g,$ and $\Theta_{r},$ there exists an open ball
\begin{equation*}
\mathbb{B}:=\mathbb{B}((\hat{s},\hat{t}),\delta)\subset \mathbb{R}^{n}\times(0,+\infty)
\end{equation*}
such that, for all $(s,t)\in\mathbb{B},$
\begin{equation*}
t\,h_{i}\left(\frac{s}{t}\right)<0,\quad i=1,\ldots,m, \quad 1-t\,g\left(\frac{s}{t}\right)<0, \quad \Theta_{r}\left(\frac{s}{t}\right)<1.
\end{equation*}

Let $\mathcal B(\mathbb R^n)$ and $\mathcal B(\mathbb B)$ denote the Borel $\sigma$-algebras on $\mathbb R^n$ and $\mathbb B\subset\mathbb R^n\times(0,\infty)$, respectively.
Let $\mu$ be the normalized Lebesgue measure on $\mathbb{B}$ (i.e., $\mu(\mathbb{B})=1$).
Define the measurable mapping $\pi:\mathbb{B}\to\mathbb R^n$ by
\begin{equation*}
\pi(s,t):=\frac{s}{t}
\end{equation*}
and the positive Borel measure $\rho$ on $\mathbb{B}$ by $d\rho:=t\,d\mu$.
Note that since $\pi(s,t)=s/t$ is continuous on $\mathbb B$, it is $(\mathcal B(\mathbb B),\mathcal B(\mathbb R^n))$-measurable.
Moreover, define $\nu:=\rho\circ \pi^{-1}$ by
\begin{equation*}
\nu(E):=\rho(\pi^{-1}(E)), \quad E\in\mathcal{B}(\mathbb{R}^n).
\end{equation*}
Then, by \cite[Theorem~3.6.1]{Bogachev2007}, define $y=(y_\alpha)_{\alpha\in\mathbb N_{2r}^n}$ by
\begin{equation*}
y_\alpha:=\int_{\pi(\mathbb{B})}x^\alpha\,d\nu(x)=\int_{\mathbb{B}}t\left(\frac{s}{t}\right)^{\alpha}d\mu(s,t), \quad \alpha\in\mathbb{N}^n_{2r}.
\end{equation*}
Then we have
\begin{align*}
L_y(h_i)
&=\sum_{\alpha\in\mathbb{N}^n_{2r}}y_\alpha(h_i)_\alpha=\int_{\mathbb{B}}t\,h_i(s/t)\,d\mu(s,t)<0, \quad i=1,\ldots,m,\\
1-L_y(g)&=1-\sum_{\alpha\in\mathbb{N}^n_{2r}}y_\alpha g_\alpha=\int_{\mathbb{B}}\left(1-t\,g(s/t)\right)d\mu(s,t)<0.
\end{align*}

Define
\begin{equation*}
\bar{t}:=\int_{\mathbb{B}}t\,d\mu(s,t).
\end{equation*}
Since $\Theta_{r}(s/t)<1$ for all $(s,t)\in\mathbb{B},$ we have
\begin{equation*}
L_y(1+\Theta_r)=\int_{\mathbb{B}}t\left(1+\Theta_r(s/t)\right)d\mu(s,t)<2\int_{\mathbb{B}}t\,d\mu(s,t)=2\bar{t}.
\end{equation*}
Since
\begin{equation*}
\frac{2}{g(\hat{x})}\leq\frac{2}{g_{\min}}<C,
\end{equation*}
we can choose $\varepsilon_{0}>0$ and $\delta>0$ sufficiently small so that $2\bar{t}<C$.
Hence,
\begin{equation*}
L_y(1+\Theta_r)<2\bar{t}<C.
\end{equation*}

Moreover, let $\mathbf{q}\in\mathbb{R}^{s(n,r)}\setminus\{0\}$ and define
\begin{equation*}
q(x):=\mathbf{q}^{T}\lceil x\rceil_r.
\end{equation*}
Then
\begin{equation*}
\mathbf{q}^{T}\left(\mathbf{M}_r(y)\right)\mathbf{q}=L_y(q^2)=\int_{\mathbb{B}}t\left\{q\left(\frac{s}{t}\right)\right\}^{2}d\mu(s,t)>0.
\end{equation*}
Indeed, since $\mathbb{B}$ is open and $q$ is a nonzero polynomial, $q(s/t)$ cannot vanish almost everywhere on $\mathbb{B}$.
Therefore,
\begin{equation*}
\mathbf{M}_r(y)\succ0.
\end{equation*}
Thus, $y$ is a strictly feasible solution to the problem~\eqref{Q_r}.
It follows from the strong duality for semidefinite programming \cite[Theorem~2.2]{Klerk2002} that
\begin{equation*}
\max\eqref{D_r}=\inf\eqref{Q_r}.
\end{equation*}

To prove the second assertion, let $\bar{y}$ be an optimal solution to the problem~\eqref{Q_r} with $\bar{y}_{0}\neq0,$ and suppose that
\begin{equation*}
\operatorname{rank}\mathbf{M}_r(\bar y)=1.
\end{equation*}
Since $\mathbf{M}_r(\bar y)\succeq0$, its $(0,0)$-entry satisfies
\begin{equation*}
\bar{y}_{0}=\left[\mathbf{M}_r(\bar y)\right]_{0,0}\geq0.
\end{equation*}
Since $\bar{y}_{0}\neq0,$ it follows that $\bar{y}_{0}>0.$

Since $\mathbf{M}_r(\bar y)$ is a rank-one positive semidefinite moment matrix, its Hankel structure implies that
\begin{equation*}
\bar{y}_{\alpha}=\bar{y}_{0}\bar{x}^{\alpha},\quad \alpha\in\mathbb{N}_{2r}^{n},
\end{equation*}
where
\begin{equation*}
\bar{x}_{i}:=\frac{\bar{y}_{e_i}}{\bar{y}_{0}},\quad i=1,\ldots,n.
\end{equation*}
Since $\bar{y}$ is feasible for the problem~\eqref{Q_r}, for each $i=1,\ldots,m,$
\begin{equation*}
L_{\bar y}(h_i)=\sum_{\alpha\in\mathbb{N}_{2r}^{n}}\bar{y}_{\alpha}(h_i)_{\alpha}=\bar{y}_{0}h_i(\bar{x})\leq0.
\end{equation*}
Since $\bar{y}_{0}>0,$ we obtain
\begin{equation*}
h_i(\bar{x})\leq0, \quad i=1,\ldots,m.
\end{equation*}
Hence, $\bar{x}$ is a feasible solution to the problem~\eqref{FP}.

From the constraint in~\eqref{Q_r}, we have
\begin{equation*}
1-L_{\bar y}(g)=1-\sum_{\alpha\in\mathbb{N}_{2r}^{n}}\bar{y}_{\alpha}g_{\alpha}=1-\bar{y}_{0}g(\bar{x})\leq0.
\end{equation*}
Thus, we have
\begin{equation*}
\bar{y}_{0}\geq\frac{1}{g(\bar{x})}.
\end{equation*}
This implies that
\begin{equation*}
\inf\eqref{FP}\geq\inf\eqref{Q_r}=L_{\bar y}(f)=\sum_{\alpha\in\mathbb{N}_{2r}^{n}}\bar{y}_{\alpha}f_{\alpha}=\bar{y}_{0}f(\bar{x})\geq\frac{f(\bar{x})}{g(\bar{x})}\geq\inf\eqref{FP}.
\end{equation*}
Therefore,
\begin{equation*}
\bar{x}:=\frac{1}{\bar{y}_{0}}\left(L_{\bar{y}}\left(x_{1}\right),\ldots,L_{\bar{y}}\left(x_{n}\right)\right)
\end{equation*}
is an optimal solution to the problem~\eqref{FP}.
\end{proof}

\subsection{A simple illustrative example}

Let $\psi:\mathbb R^2\to\mathbb R$ be defined by
\begin{equation*}
\psi(x_1,x_2):=p(x_1,x_2,1-\frac{1}{2}x_2),
\end{equation*}
where
\begin{align*}
p(x_1,x_2,x_3):={}&77x_1^6-155x_1^5x_2+445x_1^4x_2^2+76x_1^3x_2^3+556x_1^2x_2^4+68x_1x_2^5+240x_2^6\\
&-9x_1^5x_3-1129x_1^3x_2^2x_3+62x_1^2x_2^3x_3+1206x_1x_2^4x_3-343x_2^5x_3\\
&+363x_1^4x_3^2+773x_1^3x_2x_3^2+891x_1^2x_2^2x_3^2-869x_1x_2^3x_3^2+1043x_2^4x_3^2\\
&-14x_1^3x_3^3-1108x_1^2x_2x_3^3-216x_1x_2^2x_3^3-839x_2^3x_3^3+721x_1^2x_3^4\\
&+436x_1x_2x_3^4+378x_2^2x_3^4+48x_1x_3^5-97x_2x_3^5+89x_3^6.
\end{align*}
is the polynomial given in \cite[Theorem 5.8]{Ahmadi2013}.
By \cite[Theorem 5.8]{Ahmadi2013}, $\psi$ is convex but not SOS-convex.

Let $x^*:=(0,0)$ and define
\begin{equation*}
\phi(x):=\psi(x)-\psi(x^*)-\nabla\psi(x^*)^{T}(x-x^*).
\end{equation*}
Since $\psi$ is convex, its first-order characterization gives $\phi(x)\geq 0$ for all $x\in\mathbb R^2$.
Moreover, $\phi(x^*)=0$ and $\nabla\phi(x^*)=0$.
Since adding an affine polynomial does not change the Hessian, $\phi$ is convex but not SOS-convex.

Consider the following problem:
\begin{equation}\label{FP1}
\min_{x\in\mathbb R^2}\left\{\frac{f(x)}{g(x)}\mid h_i(x)\leq 0,\ i=1,2,3\right\},\tag{FP$_1$}
\end{equation}
where $f(x):=8+\phi(x),$ $g(x):=8-x_1^2-x_2^2,$ $h_1(x):=x_1^2+x_2^2-4,$ $h_2(x):=-x_1,$ and $h_3(x):=-x_2.$
Clearly, the feasible set is given by
\begin{equation*}
K:=\left\{x\in\mathbb R^2\mid x_1^2+x_2^2-4\leq 0,\ -x_1\leq 0,\ -x_2\leq 0\right\}.
\end{equation*}
The polynomials $f$, $-g$, and $h_i$, $i=1,2,3$, are convex, while $f$ is not SOS-convex.
Moreover, $f(x)=8+\phi(x)\geq 8>0$ and $g(x)>0$ for all $x\in K$.

Let $\hat{x}:=(1/2,1/2)$.
Then
\begin{equation*}
h_1(\hat{x})=-\frac{7}{2}<0, \quad h_2(\hat{x})=-\frac12<0, \quad h_3(\hat{x})=-\frac12<0,
\end{equation*}
and so, the Slater condition is satisfied.

Now, we apply the SDP relaxation established in Section~\ref{sect::4}.
Since $\deg f=6$, we take $r\geq 3$.
Moreover, since $\max_{x\in K}\|x\|=2$ and $\min_{x\in K}g(x)=4,$ we choose $L=3$ and $C=1>2/4=1/2$.
Thus, 
\begin{equation*}
\Theta_r(x)=\frac{x_1^{2r}+x_2^{2r}}{3^{2r}}.
\end{equation*}

We now consider the following SOS relaxation:
\begin{align}\label{D1}
\sup_{\lambda_i,\gamma,\eta}\quad &\gamma-\eta \tag{D$_1$}\\
{\rm s.t.}\quad & f+\lambda_1h_1+\lambda_2h_2+\lambda_3h_3-\gamma g+\eta(1+\Theta_r)\in\Sigma[x_1,x_2]_{2r},\nonumber\\
&\lambda_i\geq0,\ i=1,2,3, \ \gamma\geq0,\ \eta\geq0.\nonumber
\end{align}

Starting from the first level $r=d=3$, the perturbation term is
\begin{equation*}
\Theta_3(x)=\frac{x_1^6+x_2^6}{729}.
\end{equation*}
Then, the corresponding Lagrangian dual of problem~\eqref{D1} at $r=3$ is as follows:
\begin{align}\label{Q1}
\inf_{y\in\mathbb R^{s(2,6)}}\quad &8y_{00}+L_y(\phi) \tag{Q$_1$}\\
{\rm s.t.}\quad &1+y_{20}+y_{02}-8y_{00}\leq0, \nonumber\\
&y_{20}+y_{02}-4y_{00}\leq0, \nonumber\\
&-y_{10}\leq0, \ -y_{01}\leq0,\nonumber\\
&y_{00}+\frac{y_{60}+y_{06}}{729}\leq1,\nonumber\\
&\mathbf M_3(y)\succeq0.\nonumber
\end{align}

Using CVX \cite{Grant2013CVX} in MATLAB, we obtained an approximate optimal value of $1$, and an approximate optimal solution $\bar y$ for problem~\eqref{Q1}.
In particular,
\begin{equation*}
\bar y_{00}\approx 0.1250, \quad \bar y_{10}\approx0, \quad \bar y_{01}\approx0.
\end{equation*}
The moment matrix $\mathbf M_3(\bar y)$ was numerically observed to have rank one, indicating numerical satisfaction of the rank-one condition in Theorem~\ref{Thm4.4}. 
The extracted point is
\begin{equation*}
\bar x:=\left(\frac{\bar y_{10}}{\bar y_{00}},\frac{\bar y_{01}}{\bar y_{00}}\right)\approx(0,0).
\end{equation*}

Motivated by this numerical extraction, we next verify analytically that the candidate point $\bar x=(0,0)$ is an exact optimal solution to problem~\eqref{FP1}.

Since $\phi(x)\geq0$ for all $x\in\mathbb R^2$, we have
\begin{equation*}
f(x)=8+\phi(x)\geq8
\end{equation*}
for all $x\in K$.
Moreover, since
\begin{equation*}
g(x)=8-x_1^2-x_2^2\leq8
\end{equation*}
for all $x\in K$ and $g(x)>0$ on $K$, it follows that
\begin{equation*}
\frac{f(x)}{g(x)}\geq\frac{8}{g(x)}\geq\frac{8}{8}=1, \ x\in K.
\end{equation*}
On the other hand, at $\bar x=(0,0)$, we have $\phi(\bar x)=0$ and hence
\begin{equation*}
\frac{f(\bar x)}{g(\bar x)}=1.
\end{equation*}
Therefore, $\bar x=(0,0)$ is an optimal solution to problem~\eqref{FP1}, with optimal value $1$.

\subsection{The SOS-convex case}\label{sos-convex}

In this subsection, we assume that all functions $f,-g,h_1,\ldots,h_m$ are SOS-convex polynomials.
Then the perturbation term is no longer needed, and exactness is attained at the first level of the hierarchy.
We present the related specialized problems and results for this case.

We now consider problem~\eqref{FP} under the additional assumption that $f,-g,h_1,\ldots,h_m$ are SOS-convex.
Under assumptions {\bf (A1)} and {\bf (A2)}, consider the following SOS relaxation of problem~\eqref{LD}:
\begin{align}\label{D}
\sup_{\lambda, \gamma} \quad & \gamma \tag{${\rm {D}^\prime}$} \\
\mathrm{s.t.} \quad
& f + \sum_{i=1}^{m} \lambda_i h_i - \gamma\, g \in \Sigma[x]_{2d},\nonumber \\
& \lambda_i \ge 0,\ i = 1, \ldots, m, \, \gamma \ge 0,\nonumber
\end{align}
and its Lagrangian dual reads:
\begin{align}\label{Q}
\inf_{y \in \mathbb{R}^{s(n,2d)}} \quad
& L_y(f) \tag{${\rm {Q}^\prime}$} \\
\rm{s.t.} \hspace{8mm}
& 1 - L_y(g) \le 0, \nonumber \\
& L_y(h_i) \le 0,\, i = 1, \ldots, m,  \nonumber \\
& \mathbf{M}_d(y)\succeq 0. \nonumber
\end{align}

As a direct consequence of the convex case, we obtain the following strong duality and solution extraction results for the SOS-convex case.
Since the perturbation term $\Theta_r$ is no longer needed, the proof uses arguments similar to those of Theorems~\ref{Thm4.3} and~\ref{Thm4.4}.
We provide the details for completeness.
\begin{theorem}
Let $f,-g, h_{1}, \ldots, h_{m}: \mathbb{R}^{n} \rightarrow \mathbb{R}$ be SOS-convex polynomials of degree at most $2d.$
Suppose that assumptions {\rm {\bf (A1)}} and {\rm {\bf (A2)}} hold.
If the problem~\eqref{FP} satisfies the Slater condition, then
\begin{equation*}
\inf \eqref{FP} =\max \eqref{D}=\inf \eqref{Q}.
\end{equation*}
Moreover$,$ if $\bar{y}$ is an optimal solution to the problem~\eqref{Q} with $\bar{y}_{0} \neq 0,$ then
\begin{equation*}
\bar{x}:=\frac{1}{\bar{y}_{0}}\left(L_{\bar{y}}\left(x_{1}\right),\ldots,L_{\bar{y}}\left(x_{n}\right)\right)
\end{equation*}
is an optimal solution to the problem~\eqref{FP}.
\end{theorem}

\begin{proof}
By Theorem~\ref{Thm4.2} and Lemmas~\ref{lemma3.1} and~\ref{lemma3.2}, we have
\begin{equation*}
\max\eqref{LD}=\inf\eqref{FP}.
\end{equation*}
Let $(\lambda^*,\gamma^*)$ be an optimal solution to problem~\eqref{LD}, and define
\begin{equation*}
p:=f+\sum_{i=1}^{m}\lambda_i^*h_i-\gamma^*g.
\end{equation*}
Then $p$ is a globally nonnegative SOS-convex polynomial.
By Lemma~\ref{lemma2.4}, $p$ attains its minimum at some $u\in\mathbb{R}^n$.
Hence,
\begin{equation*}
p(u)\geq0 \quad\text{and}\quad \nabla p(u)=0.
\end{equation*}
Since $p-p(u)$ is SOS-convex and satisfies
\begin{equation*}
\left(p-p(u)\right)(u)=0 \quad\text{and}\quad \nabla\left(p-p(u)\right)(u)=0,
\end{equation*}
it follows from \cite[Lemma~8]{Helton2010} that $p-p(u)$ is SOS.
Since $p(u)\geq0$, we conclude that $p$ is SOS.
Therefore, $(\lambda^*,\gamma^*)$ is feasible for problem~\eqref{D}, and hence
\begin{equation*}
\max\eqref{D}\geq\gamma^*=\max\eqref{LD}.
\end{equation*}
On the other hand, every feasible solution to problem~\eqref{D} is feasible for problem~\eqref{LD}, and thus
\begin{equation*}
\max\eqref{D}\leq\max\eqref{LD}.
\end{equation*}
Consequently,
\begin{equation*}
\max\eqref{D}=\max\eqref{LD}=\inf\eqref{FP}.
\end{equation*}

By the same strict-feasibility argument as in the proof of Theorem~\ref{Thm4.4}, with the perturbation constraint omitted, problem~\eqref{Q} admits a strictly feasible solution.
It follows from the strong duality for semidefinite programming
\cite[Theorem~2.2]{Klerk2002} that
\begin{equation*}
\max\eqref{D}=\inf\eqref{Q}.
\end{equation*}

To prove the solution extraction, let $\bar{y}$ be an optimal solution to the problem~\eqref{Q}.
Since
\begin{equation*}
\mathbf{M}_d(\bar y)\succeq0,
\end{equation*}
its $(0,0)$-entry satisfies $\bar y_0\geq0$.
Since $\bar y_0\neq0$, we have $\bar y_0>0$.
Then, we have
\begin{align*}
& 1-L_{\bar y}(g)=1-\sum_{\alpha} \bar{y}_{\alpha} g_{\alpha} \leq 0 ,\\
& L_{\bar y}(h_i)=\sum_{\alpha} \bar{y}_{\alpha} (h_{i})_{\alpha} \leq 0, \quad i=1, \ldots, m, \\
& \mathbf{M}_d(\bar y) \succeq 0.
\end{align*}
Putting
\begin{equation*}
\bar z_\alpha:=\frac{\bar y_\alpha}{\bar y_0},\quad \alpha\in\mathbb N_{2d}^n,
\end{equation*}
we have
\begin{align}
&\frac{1}{\bar y_0}-L_{\bar z}(g)\leq0, \label{11}\\
&L_{\bar z}(h_i)\leq0, \quad i=1,\ldots,m, \label{12}\\
&\mathbf M_d(\bar z)\succeq0,\quad \bar z_0=1. \label{13}
\end{align}
It follows from~\eqref{12} that for each $i=1,\ldots,m,$
\begin{equation*}
0\geq\sum_{\alpha}\bar{z}_{\alpha}(h_i)_\alpha=L_{\bar z}(h_i).
\end{equation*}
Since $\bar z$ satisfies~\eqref{13}, by Lemma~\ref{lemma2.3} we see that for each $i=1,\ldots,m,$
\begin{equation*}
0\geq L_{\bar z}(h_i)\geq h_i\left(L_{\bar z}(x_1),\ldots,L_{\bar z}(x_n)\right)=h_i(\bar x).
\end{equation*}
So, $\bar{x}$ is a feasible solution to the problem~\eqref{FP}.

Moreover, applying Lemma~\ref{lemma2.3} to the SOS-convex polynomial $-g$, we have
\begin{equation*}
L_{\bar z}(-g)\geq-g(\bar x),
\end{equation*}
and hence, $L_{\bar z}(g)\leq g(\bar x)$.
Therefore, by~\eqref{11}, we have
\begin{equation*}
\frac{1}{\bar{y}_{0}}\leq\sum_{\alpha\in\mathbb{N}_{2d}^{n}}\bar{z}_{\alpha}g_{\alpha}=L_{\bar z}(g)\leq g(\bar x),
\end{equation*}
i.e.,
\begin{equation*}
\bar y_0\geq\frac{1}{g(\bar x)}.
\end{equation*}

Applying Lemma~\ref{lemma2.3} to $f$, we obtain
\begin{equation*}
L_{\bar z}(f)\geq f(\bar x).
\end{equation*}
This implies that
\begin{equation*}
\inf\eqref{Q}=L_{\bar y}(f)=\bar y_0L_{\bar z}(f)\geq\bar y_0f(\bar x)\geq\frac{f(\bar x)}{g(\bar x)}\geq\inf\eqref{FP}=\inf\eqref{Q}.
\end{equation*}
Thus,
\begin{equation*}
\bar{x}:=\frac{1}{\bar{y}_{0}}\left(L_{\bar{y}}\left(x_{1}\right),\ldots,L_{\bar{y}}\left(x_{n}\right)\right)
\end{equation*}
is an optimal solution to the problem~\eqref{FP}.
\end{proof}

\section{Conclusions}\label{sect::5}

Under the Slater condition, asymptotic convergence and solution extraction under a rank-one condition have been established.
As an important special case, we have shown that the perturbation term is unnecessary for SOS-convex polynomial data and that the first SDP relaxation is exact.
It would be interesting to extend the CCT-based approach to fractional programs with nonconvex polynomial data.
This is left for future research.




\subsection*{Acknowledgments}
A part of this work was completed while the three authors were visiting the Vietnam Institute for Advanced Study in Mathematics (VIASM).
They are grateful to the Institute for its hospitality and support.
The authors also wish to thank Prof. Feng Guo of Dalian University of Technology for suggesting the use of the high-degree perturbation method and for many fruitful discussions.

\subsection*{Availability of data and materials}
No datasets were generated in this study.


\subsection*{Conflict of interest}
The authors have no conflicts of interest to declare.

\small


\begin{thebibliography}{10}

\bibitem{Ahmadi2013}
A.~A. Ahmadi and P.~A. Parrilo.
\newblock A complete characterization of the gap between convexity and
  {SOS}-convexity.
\newblock {\em SIAM J. Optim.}, 23(2):811--833, 2013.

\bibitem{Bajalinov2003}
E.~Bajalinov.
\newblock {\em Linear-Fractional Programming: Theory, Methods, Applications and
  Software}.
\newblock Springer, 2003.

\bibitem{Belousov2002}
E.~G. Belousov and D.~Klatte.
\newblock A {Frank}–{Wolfe} type theorem for convex polynomial programs.
\newblock {\em Comput. Optim. Appl.}, 22:37--48, 2002.

\bibitem{Bogachev2007}
V.~I. Bogachev.
\newblock {\em Measure Theory}.
\newblock Springer Berlin, Heidelberg, 2007.

\bibitem{Bot2017}
R.~I. Bot and E.~R. Csetnek.
\newblock Proximal-gradient algorithms for fractional programming.
\newblock {\em Optimization}, 66(8):1383--1396, 2017.

\bibitem{Bot2}
R.~I. Bot, M.~N. Dao, and G.~Li.
\newblock Extrapolated proximal subgradient algorithms for nonconvex and
  nonsmooth fractional programs.
\newblock {\em Math. Oper. Res.}, 47(3):2415--2443, 2022.

\bibitem{Bot3}
R.~I. Bot, M.~N. Dao, and G.~Li.
\newblock Inertial proximal block coordinate method for a class of nonsmooth
  sum-of-ratios optimization problems.
\newblock {\em SIAM J. Optim.}, 33(2):361--393, 2023.

\bibitem{Bot1}
R.~I. Bot, G.~Li, and M.~Tao.
\newblock Full splitting algorithms for fractional programs with structured
  numerators and denominators.
\newblock {\em SIAM J. Optim.}, 35(4):2623--2653, 2025.

\bibitem{Boyd2004}
S.~Boyd and L.~Vandenberghe.
\newblock {\em Convex Optimization}.
\newblock Cambridge University Press, 2004.

\bibitem{MR152370}
A.~Charnes and W.~W. Cooper.
\newblock Programming with linear fractional functionals.
\newblock {\em Naval Res. Logist. Quart.}, 9:181--186, 1962.

\bibitem{Chen2005JOTA}
J.~C. Chen, H.~C. Lai, and S.~Schaible.
\newblock Complex fractional programming and the {C}harnes--{C}ooper
  transformation.
\newblock {\em J. Optim. Theory Appl.}, 126(1):203--213, 2005.

\bibitem{Dinkelbach1967}
W.~Dinkelbach.
\newblock On nonlinear fractional programming.
\newblock {\em Management Sci.}, 13(7):492--498, 1967.

\bibitem{Grant2013CVX}
M.~C. Grant and S.~Boyd.
\newblock {\em The {CVX} user's guide, release 2.0}, 2013.
\newblock User manual, \url{http://cvxr.com/cvx/}.

\bibitem{MR4810562}
C.~G\"unther, A.~Orzan, and R.~Precup.
\newblock Componentwise {D}inkelbach algorithm for nonlinear fractional
  optimization problems.
\newblock {\em Optimization}, 73(11):3323--3337, 2024.

\bibitem{Guo2021}
F.~Guo and L.~G. Jiao.
\newblock On solving a class of fractional semi-infinite polynomial
  programming problems.
\newblock {\em Comput. Optim. Appl.}, 80(2):439--481, 2021.

\bibitem{Guo2020}
F.~Guo and X.~X. Sun.
\newblock On semi-infinite systems of convex polynomial inequalities and
  polynomial optimization problems.
\newblock {\em Comput. Optim. Appl.}, 75(3):669--699, 2020.

\bibitem{Guo2024}
F.~Guo and M.~Zhang.
\newblock An {SDP} method for fractional semi-infinite programming problems
  with {SOS}-convex polynomials.
\newblock {\em Optim. Lett.}, 18(1):105--133, 2024.

\bibitem{Helton2010}
J.~W. Helton and J.~Nie.
\newblock Semidefinite representation of convex sets.
\newblock {\em Math. Program.}, 122(1):21--64, 2010.

\bibitem{Jiao2019}
L.~G. Jiao and J.~H. Lee.
\newblock Fractional optimization problems with support functions: exact {SDP}
  relaxations.
\newblock {\em Linear Nonlinear Anal.}, 5(2):255--268, 2019.

\bibitem{Klerk2002}
E.~Klerk.
\newblock {\em Aspects of Semidefinite Programming}.
\newblock Springer, 2002.

\bibitem{Lasserre2015}
J.-B. Lasserre.
\newblock {\em An Introduction to Polynomial and Semi-Algebraic Optimization}.
\newblock Cambridge University Press, 2015.

\bibitem{Lasserre2007}
J.-B. Lasserre and T.~Netzer.
\newblock {SOS} approximations of nonnegative polynomials via simple high degree perturbations.
\newblock {\em Math. Z.}, 256(1):99--112, 2007.

\bibitem{Nguyen2015}
V.-B. Nguyen, R.-L. Sheu, and Y.~Xia.
\newblock An {SDP} approach for solving quadratic fractional programming
  problems.
\newblock {\em Optim. Methods Softw.}, 31(4):701--719, 2015.

\bibitem{Reznick2000}
B.~Reznick.
\newblock Some concrete aspects of {H}ilbert's 17th {P}roblem.
\newblock In {\em Real algebraic geometry and ordered structures ({B}aton
  {R}ouge, {LA}, 1996)}, volume 253 of {\em Contemp. Math.}, pages 251--272.
  Amer. Math. Soc., Providence, RI, 2000.

\bibitem{schaible1974}
S.~Schaible.
\newblock Parameter-free convex equivalent and dual programs of fractional
  programming problems.
\newblock {\em Z. Operations Res. Ser. A-B}, 18:A187--A196, 1974.

\bibitem{Schaible2003}
S.~Schaible and J.~Shi.
\newblock Fractional programming: the sum-of-ratios case.
\newblock {\em Optim. Methods Softw.}, 18(2):219--229, 2003.

\bibitem{StancuMinisian2012}
I.~M. Stancu-Minisian.
\newblock {\em Fractional Programming: Theory, Methods and Applications},
  volume 409 of {\em Lecture Notes in Economics and Mathematical Systems}.
\newblock Springer, 2012.

\bibitem{Yang2024}
C.~M. Yang, L.~G. Jiao, and J.~H. Lee.
\newblock A parameter-free approach for solving {SOS}-convex semi-algebraic
  fractional programs.
\newblock {\em Optimization}, 2026.
\newblock \url{https://doi.org/10.1080/02331934.2026.2642340}.


\end{thebibliography}

\end{document}